\documentclass[10pt,reqno]{amsart}
     \makeatletter
     \def\section{\@startsection{section}{1}%
     \z@{.7\linespacing\@plus\linespacing}{.5\linespacing}%
     {\bfseries
     \centering
     }}
     \def\@secnumfont{\bfseries}
     \makeatother

\newtheorem{theorem}{Theorem}[section]
\newtheorem{lemma}[theorem]{Lemma}
\newtheorem{proposition}[theorem]{Proposition}
\newtheorem{corollary}[theorem]{Corollary}
\theoremstyle{definition}
\newtheorem{definition}[theorem]{Definition}

\theoremstyle{remark}
\newtheorem{remark}[theorem]{Remark}
\numberwithin{equation}{section}

\makeatletter
\newcommand\@received{Received 2024-3-4;
   Accepted 2024-6-2; Communicated by the editors.}
\renewcommand\@adminfootnotes
{\let \@makefnmark \relax \let \@thefnmark \relax
\@footnotetext{\@received }
\ifx \@empty \@date \else \@footnotetext {\@setdate }\fi
\ifx \@empty \@subjclass \else \@footnotetext {\@setsubjclass }\fi
\ifx \@empty \@keywords \else \@footnotetext {\@setkeywords }\fi
\ifx \@empty \thankses \else \@footnotetext
{\def \par {\let \par \@par }\@setthanks }\fi}
\makeatother

\usepackage{url}
\usepackage{xcolor}

 \usepackage{type1cm}        

\usepackage{graphicx}        
 \usepackage{subcaption}
 \usepackage[export]{adjustbox}
\usepackage{hyperref}
\usepackage{eulervm}

   \usepackage{latexsym, amsmath}
  \usepackage{amsfonts, amssymb}
 \usepackage{relsize}

\def\E{{\ensuremath{\mathbb {E}}}}

\def\R{\ensuremath{\mathbb {R}}}
\def\C{\ensuremath{\mathbb {C}}}

\def\Z{\ensuremath{\mathbb {Z}}}

\def\Pp{\ensuremath{\mathbb {P}}}

\def\si{\sigma}

\def\mfB{\mathfrak{B}}

\def\la{\lambda}

\def\({\left(}

\def\){\right)}
\def\al{\alpha}
\def\vf{\varphi}

\def\vep{\varepsilon}
\def\abs#1{\vert\, #1\,\vert}

\def\be{\beta}

\def\ga{\gamma}

\def\mcS{{\mathcal S}}

\def\mcD{{\mathcal D}}
\def\mcN{{\mathcal N}}

\def\mcU{{\mathcal U}}

\def\mcW{{\mathcal W}}

\def\mcI{{\mathcal I}}
\def\mcJ{{\mathcal J}}

\def\mcF{{\mathcal F}}

\def\d{\,\mathrm{d}}
\def\iy{\infty}
\def\bydef{:=}

\def\beq{\begin{equation}}
\def\eeq{\end{equation}}
\def\bei{\begin{itemize}}
\def\eei{\end{itemize}}

\def\trp{(\Omega, \mfB, \Pp)}
\def\gLf{\text{generalized Lambert function}}
\def\dar{\downarrow}

\begin{document}
\title[On the Generalized Lambert Function]{On the Generalized Lambert Function} 
 
\author[Alexander Kreinin]{Alexander Kreinin*}
\thanks{* Corresponding author}
\address{Alexander Kreinin: Department of Computer Science, University of Toronto, Toronto, ON, M5S 2E4, CANADA}
\email{alexander.kreinin@utoronto.ca}

\author[Andrey Marchenko]{Andrey Marchenko}
\address{Andrey Marchenko: LayerOneFinancial}
\email{avmarch@gmail.com}

\author[Vladimir V. Vinogradov]{Vladimir V. Vinogradov}
\address{Vladimir V. Vinogradov: Department of Mathematics, Ohio University, Athens, OH 45701 USA}
 \email{vinograd@ohio.edu}
 \address{Vladimir V. Vinogradov:
 Department of Computer and Mathematical Sciences, IA 4118,
University of Toronto Scarborough, 1265 Military Trail,
Toronto, Ontario M1C 1A4, Canada}
\email{vv.vinogradov@utoronto.ca}
 \date{27 March, 2025}

  \subjclass[2020]{Primary 60E05, 60F05; Secondary 60E07}
\keywords{Continued exponential representation, Galton-Watson process, generalized Lambert function, iteration process, 
 moment structure, Monte Carlo simulation, probability--generating function}

\begin{abstract}
We   consider a particular generalized Lambert function, $y(x)$, defined by the implicit equation
$ y^\be =  1 - e^{-xy}$, with $x>0$ and $ \be >  1$.  
Solutions to this equation can be found in terms of a certain continued exponential defined by formula (\ref{eq:gL_cr}).
 Asymptotic and structural properties of a non-trivial
  solution, $y_\be(x)$, and its connection to the extinction probability of related
branching processes are discussed. We demonstrate that this function constitutes  a cumulative distribution function of a 
previously unknown non-negative absolutely continuous random variable. 
\end{abstract}

\maketitle

\section{Introduction}
The Lambert $W$ function has a long history going back to $18^{th}$ century. Nevertheless, it was
rigorously introduced in 1996 only and studied relatively recently. 
This function is closely related to the logarithmic function and arises  in
 many models of the different areas of  sciences, including a  number of problems in ecology and evolution theory.
 The main reason why the Lambert $W$ function became so popular is because this function allows us to solve explicitly 
 important models for which this is not possible in terms of elementary functions. 
  
The present paper is partly devoted to the interplays of the classical analysis, in particular, transcendental functions, 
probability theory  as well
as certain areas of computational mathematics, namely iteration schemes,  Monte Carlo simulation
and evolution of dynamical systems.

The classical Lambert $W$ function is  defined as (the multi-valued) inverse function to
$$y=x\cdot e^x, $$ (see \cite{CG}) and therefore, it satisfies the following implicit equation:
$$ W(x)\cdot e^{W(x)}=x.  $$
This function
was introduced by Lambert \cite{La} in the second half of $18^{th}$ century and became a subject of Leonhard Euler's studies
\cite{Eu}  on the transcendental equations. Since then thousands of papers dedicated to the Lambert function and its
generalizations were published.

The function $W$,  has many applications in various areas including
 mathematical and theoretical physics  \cite{CG}, \cite{DG}, \cite{LR}, \cite{Me}. 
This function is deeply connected to many
explicitly solvable models and problems ranging from population growth rates to fertilization kinetics and disease dynamics
\cite{BWH}, \cite{LR},  \cite{MaS}, \cite{Me},  \cite{SMM}, \cite{Sct}.

The function $W$ provides an exact solution to the equation
$$e^{-cx} =a\cdot(x-b) $$
(see \cite{CG}).
The function $W_p$, the principal real branch  of the Lambert $W$ function, has the Taylor series expansion 
$$ W_p(x) = \sum_{n=1}^\iy \frac{(-n)^{n-1}}{n!} x^n,  $$
with the  radius of convergence, $\rho = e^{-1}$.
It also admits the following {\it continued exponential } representation \cite{CG}
\beq
    W(x) = x\cdot e^{-x e^{-x e^{-x\dots}}}.  \label{eq:std_W_cer}
  \eeq  
  
  Generalizations of the Lambert $W$ function were also relatively recently considered.
  In particular,
  in quantum mechanics research on Bose-Fermi statistics the  equation
  $$
    e^{-cx}=a\cdot \frac{P_n(x)}{Q_m(x)},
   $$  where $P_n$ and $Q_m$ are particular polynomials of degrees $n$ and $m$ respectively,
  was analyzed \cite{MeBa}. In \cite{BWH}, \cite{Cas} applications of the generalized Lambert functions are
  recently seen in the area of 
  atomic, molecular, and optical physics.

  In the theory of branching processes, a generalization of the Lambert equation describes the extinction probability
  of an explosive Galton-Watson branching process. 
   
   Probabilistic flavour of the solution to the generalized Lambert equation (see Section~\ref{sec:ext_prob})
    naturally leads to an interesting family of
   random variables that we denote $\xi_\be$ in what follows.
    The moments of $\xi_\be$  can be computed in closed form. The study of the analytical properties of the moments 
    lead to an elegant area of mathematics called Tornheim series (see \cite{Torn} as well as
    \cite{AliD},    \cite{BaB},  \cite{Bor2}, \cite{Bor1},     \cite{Kub}).
   
   The structure of the present paper is as follows. In Section~\ref{sec:ext_prob} we discuss connection between the 
   generalized Lambert function and the extinction probability of the Galton-Watson process.
   In Section~\ref{sec:gLf} the main analytical properties of the \gLf are discussed.
   
  In Section~\ref{sec:asympt_gLf}, we describe the asymptotic expansions and two-sided bounds for the \gLf.
  In Section~\ref{sec:numerics}, the numerical examples of the computation of the {\gLf}   are discussed.
  A probabilistic interpretation of the generalized Lambert function  and characteristics of the
  corresponding random variables are discussed in Section~\ref{sec:rand_var}. Some technical statements are 
  deferred to  the Appendix.
   
   To conclude the Introduction, we summarize some relevant notation and terminology.
 Hereinafter, $\C$, $\R$,
$\R_+^1$, $\Z$, $\Z_+$,  stand for the sets of  all complex numbers, all real numbers,
all positive reals, all integers, and all non-negative
 integers, respectively. In what follows, we use the notation $F_\eta(x)$ for the cumulative distribution function
 (or cdf) of the random variable (r.v.) $\eta$. The symbol ``$\bar F_\eta(x)$" stands for the survival function,
 $\bar F_\eta(x)\bydef 1 -  F_\eta(x)$. The notation 
``$\E[\eta^m]$" is used for  the $m^{th}$ moment of the random variable $\eta$, $m\in \Z$. 
The natural logarithm of $x$ is denoted by $\log x$. The Gamma function is traditionally denoted by $\Gamma(x)$,
$x\in\R$, $x\ge 0$.
 The Riemann zeta function is denoted by $\zeta(z)$, $z\in\C$. If  $z$ is real and $z> 1$,
 $\zeta(z) \bydef \sum_{n=1}^\iy\limits \frac{1}{n^z}. $
\section{Extinction probability of a Galton-Watson process }\label{sec:ext_prob}
Extinction probability of a Galton-Watson process satisfies the equation
\beq
     \pi_\nu(z)=z, \quad 0\le z\le 1, \label{eq:extinc_prob}
 \eeq   
where $\pi_\nu(z)$ is the probability  generating function of the number of offsprings, $\nu$, in the branching process
(see \cite{Fel},  \cite{KarlT},  \cite{KD}).
If the  r.v. $\nu$ has a Poisson distribution with parameter $\la>0$, then Equation~(\ref{eq:extinc_prob}) 
takes the form
$$ e^{\la(z-1)}=z, $$
and can be transformed to
$$ x e^{x} = -\la\cdot e^{-\la}. $$
The connection with the classical Lambert $W$ function and the Lambert equation becomes clear in this case.

If the distribution of the r.v. $\nu$ belongs to the class of {\it discrete stable distributions} the probabilities
$ \Pp(\nu = k)$,  are defined by the equation
$$ \Pp(\nu = k) = \frac{1}{k!} \cdot 
{}_{1}\Psi_{1}(-\alpha, k; -\alpha, 0; \la), $$
where
${}_{1}\Psi_{1}$ is a particular case of the Wright function, such that 
given $\delta \in \mathbb{C}$, $\rho  \in (-1,0)\cup(0,\infty)$, argument 
$z  \in \mathbb{C}$ with $|z| < \infty$ and {\it real} $k \geq 0$, 
\begin{equation}	\label{eq:CoHyF}
{}_1\Psi_1(\rho, k; \rho, \delta; z) \bydef
  \sum_{n=0}^\infty \frac{\Gamma(\rho n+k)}{\Gamma(\rho n+\delta)} \cdot  \frac{z^n}{n!}.
\end{equation}
A discrete stable r.v., $\nu$ , 
with parameters $\la>0$ and 
$\alpha \in (0,1]$ is characterized by its p.g.f., ${\mathcal P}_\al(u)\bydef \E[u^\nu]$
\begin{equation}	\label{DiStPGF}	
{\mathcal P}_\al(u) 
= \exp\{ -\la\cdot (1-u)^{\alpha}\}.
\end{equation}
The well-known case $\alpha = 1$ corresponds to the subclass of {\it Poisson} distributions which are naturally incorporated to the
family of discrete stable laws. 
 
 The probability of ultimate extinction for a supercritical Galton-Watson process 
with discrete stable branching mechanism that starts from one particle appears to be intimately connected to a 
generalized Lambert function.
Indeed, Equation~(\ref{eq:extinc_prob}) can be written as
$$ e^{-\la (1-u )^\al} = u.  $$
Denote $x=(1-u)^\al$. Then $1-u=x^\be$ and the latter equation can be written as follows:
\beq
    e^{-\la x} =1-x^\be,   \quad x\in (0, 1], \quad \la>0, \label{eq:e_pr_ds}
 \eeq
where $\be = \al^{-1}$.

The following rather elementary assertion stipulates the existence and uniqueness of the solution to (\ref{eq:e_pr_ds}).
\begin{proposition}\label{prop:exist_sol}
Let $\la>0$ and $\be>1$. Then
Equation~{\rm{(\ref{eq:e_pr_ds})}}  has the unique solution in the open interval $(0, 1)$. 
\end{proposition}
\begin{proof} 
Indeed, consider the function 
\beq 
    {\mcF}(x)\bydef  e^{-\la x} -1+x^\be, \quad \be>1.  \label{eq:mcF_be}
 \eeq
We have
$$ \mcF(0) = 0, \quad {\mcF}(1) > 0, \quad\text{ and the derivative } \mcF^{\prime}(0) < 0. $$
The second derivative  $\mcF^{\prime\prime}(x)>0$. Therefore,  the function $\mcF$ is convex 
in the interval $[0, 1]$ and Equation~(\ref{eq:e_pr_ds}) has a unique solution, $x_0$, satisfying the inequalities
$0<x_0<1$.
\end{proof}
 \begin{figure}[th]
	\begin{center}
	\vspace{-150pt}
	\includegraphics[width=0.9\textwidth]{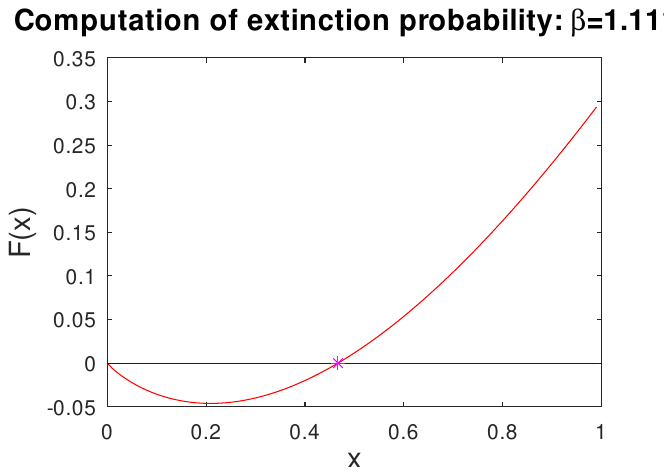}
	\vspace{-160pt}
	\end{center}
	 \caption{ The function $\mcF(x)$.}
  \label{fig:mcF}
\end{figure}
Here, the root of the equation $\mcF(x)=0$ is shown by the star in Figure~\ref{fig:mcF}. For $\be=1.11$ and $\la=1.2$,
the root $x_0\approx 0.46671$. The root of Equation~(\ref{eq:e_pr_ds}) is found by the iteration process described by 
Theorem~\ref{thm:conv_iter} in the next Section. The iteration process is stopped once the iterations, $x_n$,
satisfy the inequality
$$ \abs{x_n - x_{n+1}}\le \vep, $$ where $\vep=10^{-6}$. For the above example, 
the number of iterations $n=25$.

\section{A generalized Lambert function}\label{sec:gLf}
In this section we consider a formal definition and some analytical properties of the \gLf.
Consider the function $x_\be: (0, 1)\to \R^1_+$ defined by the formula
 \beq
        x_{\be}(y) =\begin{cases} -\frac{\log(1-y^\be)}{y}, \qquad 0<y<1,\\
	\qquad 0, \qquad \qquad\quad y=0,\end{cases}  \label{eq:inv_func}
  \eeq
where parameter $\be>1$ is  fixed.  
\begin{definition}
The inverse function to $x_{\be}$ is called the {\it generalized Lambert function}.
\end{definition}
\noindent We shall denote the generalized Lambert function by $y_\be(x)$ in what follows.

Since $x_{\be} (y)$ is analytic for $y\ne0,1$ and continuous at $y=0$, one can see that $y_{\be} (x)$ is a real-valued,
analytic function for $x>0$.
Both functions $x_{\be} (y)$ and $y_{\be} (x)$ are monotonically increasing. Indeed, the derivative of $x_{\be} (y)$ 
is equal to
$$
    \frac{\d x_{\be} (y)}{\d y}=\frac{\frac{\be\cdot y^\be}{1-y^\be}+\log({1-y^\be})}{1-y^\be}>0 \quad 
    \textrm{for} \ \   y\in (0,1).
$$ 
Then we obtain that
$$
    \frac{dy_{\be} (x)}{dx}>0 \quad \textrm{for} \ \  x>0.
$$
Therefore, the generalized Lambert function is defined uniquely for any $x>0$.
\begin{remark}
 From (\ref{eq:inv_func}) it immediately follows that the function $x_\be(y)$ is represented by the absolutely convergent
   series
   \beq
       x_\be(y) = \sum_{k=1}^\iy \frac{y^{\be k-1}}{k}, \quad 0\le y<1. \label{eq:ser_x_be}
    \eeq
   Since $\be>1$, the derivative
   $$
       \frac{\d x_\be(y)}{\d y}= \sum_{k=1}^\iy \frac{y^{\be k-2}}{k(\be k - 1)}>0
   $$
   for $y\in (0, 1)$. Note that $ \frac{\d x_\be(y)}{\d y}$ blows up at $0$ for $1<\be<2$. For $\be\ge 2$
   it is defined at $0$ by continuity.
   \end{remark}

Let $\al\bydef {1}/{\be}$.
The generalized Lambert function $y_\be (x)$ obviously satisfies the following transcendental equation:
\beq
    y = \(1 - e^{-xy}\)^\al,    \quad x>0,   \label{eq:ybe_def}
 \eeq
with $y \in (0, 1)$ and parameter $\al\in (0, 1)$.

\begin{remark} 
Equation~(\ref{eq:e_pr_ds}) is equivalent to Equation~(\ref{eq:ybe_def}) in the space of continuous functions,
${\mathbf{C}}(\R^1_+)$. The role of the variable $x$ in
(\ref{eq:ybe_def}) plays parameter $\la$, the variable $x$ in (\ref{eq:e_pr_ds}) is represented by $y$ in
(\ref{eq:ybe_def}). Equation~(\ref{eq:ybe_def}) is equivalent to
\beq
    y^\be = 1 - e^{-xy},    \quad x>0,   \label{eq:ybe}
 \eeq
\end{remark}

\begin{remark}
Equation~({\ref{eq:ybe_def}}) has infinitely many discontinuous solutions. Indeed, let $y_\be^{(c)}$ be a continuous solution of
(\ref{eq:ybe}). Then the function
 $$
        y_{\be}^{*}(y) =\begin{cases} y_\be^{(c)}(x), \qquad \,\, x\neq x_\ast,\\
	\qquad 0, \qquad \quad x=x_\ast \end{cases} 
 $$
 is also a solution to ({\ref{eq:ybe_def}}) having discontinuity at $x_\ast$.
\end{remark}
\noindent Now, consider x as a fixed parameter and solve  Equation~(\ref{eq:ybe_def}) with respect to $y$.
\begin{lemma}\label{lem:sol_exists}
 For each $x > 0$ Equation {\rm({\ref{eq:ybe_def}})} has two  non-negative solutions, one of which is always 0, and 
 the other one belongs to $(0, 1)$.
 \end{lemma}
\begin{proof} The statement of Lemma~\ref{lem:sol_exists} follows from Proposition~\ref{prop:exist_sol}.
\end{proof}
    The function $y_\be(x)$ is intimately connected  to the mapping, $\widehat\mcW : [0, 1]\to [0, 1]$, defined by the relation
   $$ \widehat\mcW y = f(y), $$
   where
   \beq
       f(y) = \(1 - e^{-x y}\)^\al, \quad x>0,      \label{eq:mapp_W}
    \eeq    
   and  $\al=1/\be$. The mapping $\widehat\mcW$ has a trivial fixed point:
   $$ y_0 = 0, \qquad \forall x> 0.  $$
   Given $x>0$, the non-trivial fixed point of $ \widehat\mcW$ represents the value of the function, $y_\be(x)$, 
   that is a solution of Equation (\ref{eq:ybe_def}). 
   
   The next statement describes the set of functions, $Y_\be=\{ y_\be(x)\}_{\be>1}, x\in \R^1_+$ and 
   delineates the probabilistic nature of the set $Y_\be$.
   
   Let $\mcD_+$ be the totality of all cumulative distribution functions of  non-negative and absolutely
    continuous random variables.
   \begin{proposition}\label{prop:cdf_y}
   For all $\be>1$ we have
   \beq
       Y_\be\subset \mcD_+.  \label{eq:subset}
     \eeq
     \end{proposition}
     \begin{proof}
     We need to prove that
     \bei
     \item[i. ] The limits   
     \beq
        \lim_{x\to 0} y_\be(x)=0,  \label{eq:x_lim1}
     \eeq  
     and
     \beq
        \lim_{x\to \iy} y_\be(x)=1.  \label{eq:x_lim2}
     \eeq  
      \item[ii. ] The function $y_\be(x)$ is monotone.
      \item[iii. ] The function $y_\be(x)$ is absolutely continuous.
     \eei 
     
     {\it Proof of {\rm i}. }
   From the definition of the inverse function, $x_\be(y)$, we derive that 
   the limits of the inverse function satisfy the relations
   \beq
        \lim_{y\to 0} x(y)=0 \quad \text{and}\quad  \lim_{y\to 1} x(y)=+\iy.  \label{eq:x_lim}
     \eeq   
     Then from (\ref{eq:x_lim}) we obtain (\ref{eq:x_lim1}) and (\ref{eq:x_lim2}).
     
      {\it Proof of {\rm ii.} }  Monotonicity of $y_\be(x)$ immediately follows from the inequality
      $$  \frac{\d y_\be(x)}{\d x} >0 \quad\text{for $x>0$}. $$
      
       {\it Proof of {\rm iii.} }  The last statement is obvious.
     
     Thus, the set of functions, $Y_\be$, is a subset of   $\mcD_+$. 
     \end{proof}
   Let us now consider a useful computational scheme for the functions $y_\be(x)$.
   The inverse function, $x_\be(y)$, can be naturally used for the computation of  $y_\be(x)$
   for different values of the parameter $\be$, see Figure~\ref{fig:gGLf1}.
   More precisely, we take a finite set 
   $$ y_{\be}^{( k)} =\frac{k}{m}, \quad k= 1, 2, 3, \dots, m-1, $$
   for sufficiently large $m$ and compute the corresponding values of $x$,
   $$  x_k= -\frac{\log\(1-(y_{\be}^{( k)})^\be\)}{y_{\be}^{( k)}}. $$
   Hereafter, we shall call the linear interpolation between the nodes $\Bigl\{\(x_k, y_{\be}^{(k)}\)\Bigr\}_{k=1}^{m-1}$, the 
   {\it benchmark approximation} of the
   function $y_\be(x)$.
   \begin{figure}[th]
	\begin{center}
	\vspace{-150pt}
	\includegraphics[width=0.9\textwidth]{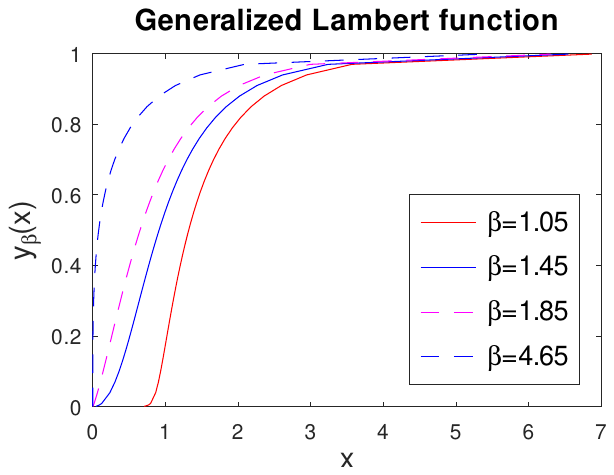}
	\vspace{-160pt}
	\end{center}
	 \caption{ Graph of the generalized Lambert function.}
  \label{fig:gGLf1}
\end{figure}

 \begin{theorem}\label{thm:conv_iter}
 The function $y_\be(x)$ admits the following  continued exponential representation:
 \beq
      y_\be(x) =  \dots\(1 -\exp\( -x\cdot \( 1 -\exp\(-x\cdot (1- e^{-x})^\al\) \)^\al \) \)^\al \dots  \label{eq:gL_cr}
 \eeq
 where $\al=\be^{-1}$.
 Formula~{\rm (\ref{eq:gL_cr})} should be understood as the limit of the iteration process 
 $$y_\be(x) \bydef \lim_{n\to \iy} y_n, $$
 where
 $$
\left\{ 
\begin{array}{ll} 
    y_{n+1} &= \( 1 - e^{-x\cdot y_n}\)^\al, \quad n=1, 2, \dots ,  \\
    y_1 &= 1. 
    \end{array}\right.
 $$
    Given  $x>0$, the fixed point of the mapping {\rm(\ref{eq:mapp_W})} is a non-trivial solution of the 
    transcendental equation {\rm (\ref{eq:ybe_def})}.
   \end{theorem}
   \begin{remark}
         It is obvious that $\forall n=1, 2, \dots, \text{we have that }  0< y_n\le 1$.
   \end{remark}
    The proof of Theorem~\ref{thm:conv_iter} is based on the following two important Lemmas~\ref{eq:lem_mon} and
    \ref{lem:contr_sol}.
   \begin{lemma}[Monotonicity  of $\widehat\mcW$] \label{eq:lem_mon}
   Consider the sequence $\{ y_n\}_{n\ge 1}$ defined by the relation $y_{n+1}=\widehat\mcW y_n$ with an arbitrary initial value,
   $0<y_1\le 1$. Then we have
   \beq
    y_1\ge y_2 \Longleftrightarrow y_n\ge y_{n+1} \quad \forall \, n=2, 3, \dots; \label{eq:gte}
    \eeq
    \beq
    y_1 \le y_2 \Longleftrightarrow y_n\le y_{n+1} \quad \forall \, n=2, 3, \dots. \label{eq:lte}
    \eeq
    Moreover, if $0<y_1 \le z_1<1$ and the sequence $\{z_n\}_{n\ge 1}$  is defined by the same mapping, $\widehat\mcW$,
    $z_{n+1}= \widehat\mcW z_n$, then
    $$ y_n\le z_n \quad n=1, 2, \dots. $$
    \end{lemma}
    \begin{proof}[Proof of Lemma]
    The proof of the lemma follows immediately from the monotonicity of the function
    $$f(y) = \(1 - e^{-xy}\)^\al, \quad 0<\al<1, \,\, x>0, $$
    that we obtain from the inequality $\frac{\d f(y)}{\d y}>0$ for $x>0$ and $ y\in(0, 1]$. 
    Suppose that $1\ge y_1>y_2>0$. Let us prove that $y_2>y_3$. But $y_2^\be = 1 - e^{-x y_1}$ and
    $y_3^\be = 1 - e^{-x y_2}$. Since $y_1>y_2$, we obtain that
    $ y_2^\be > y_3^\be$. Since $\be>1$, we derive that $y_2>y_3$. 
    
    A similar reasoning allows us to conclude that if
    $y_n > y_{n+1}$ then $y_{n+1}> y_{n+2}$, and we obtain that the inequality $y_1>y_2$
    implies that $y_n>y_{n+1}$ for all integer $n\ge 2$. 
    
    The dual case, $0<y_1< y_2\le 1$ is analogous to the previous one: the latter inequality
    implies that $y_n < y_{n+1}$ for all integer $n\ge 2$.
    
    The proof of the last statement of the Lemma is similar to the proof of the first statement.
    \end{proof}
    Next, denote $$ g(y)\bydef \frac{\d f(y)}{\d y}. $$
    \begin{lemma}[Contraction at fixed point] \label{lem:contr_sol}
    $ $
    
    \bei
    \item[i.] The fixed point, $y_\ast(x)$, of the mapping  $\widehat\mcW$ satisfies the relation
    \beq
        0<g(y_\ast)\le \al<1 \qquad\forall x>0. \label{eq:gY}
    \eeq
    \item[ii.] The  derivative, $g(y)$, diverges at $0$:
    $$\lim_{y\to 0}g(y)=+\iy. $$ 
    \item[iii.] The function $g(y)$ is monotonically decreasing. 
    \item[iv.] For $y>y_\ast$, $g(y)<\al$.
    \eei
    \end{lemma}
    \begin{proof}[Proof of Lemma~{\rm \ref{lem:contr_sol}}]
    The function, $g(y)$, satisfies the relation
    $$ g(y) = \al \cdot x\cdot \(1-e^{-yx}\)^{\al-1} e^{-yx}, $$
    and
    $$ g^\prime(y)=-\al x^2\cdot e^{-yx} \cdot  \(1 - e^{-x\cdot y_\ast}\)^{\al-1} - \al (1-\al) \cdot 
    e^{-yx} \cdot  \(1 - e^{-x\cdot y_\ast}\)^{\al-1} <0. $$
    The latter inequality implies statement ({\rm{iii}}). 
    
    Denote by $y_\ast(x)$ the fixed point of the mapping $\widehat\mcW$. Then we have
    $$ y_\ast = \(1 - e^{-x\cdot y_\ast}\)^\al, \quad x>0. $$
    Therefore, 
    \begin{eqnarray*}
    g(y_\ast) &=& \al \cdot x\cdot \(1 - e^{-x\cdot y_\ast}\)^{\al-1} \cdot e^{-x\cdot y_\ast} \\
        &=& \al \cdot x\cdot \(1 - e^{-x\cdot y_\ast}\)^{\al} \cdot \frac{e^{-x\cdot y_\ast}}{1-  e^{-x\cdot y_\ast}}\\
        &=& \al \cdot x\cdot y_\ast \cdot  \frac{e^{-x\cdot y_\ast}}{1-  e^{-x\cdot y_\ast}}
    \end{eqnarray*}
    Denote $t\bydef x\cdot y_\ast $. Obviously, $t>0$. Notice that for $t>0$,
    $$ e^t -1 >t.$$
    Then we obtain that $g(y_\ast)$ satisfies the inequality
    $$ g(y_\ast) = \al \cdot \frac{t}{e^t -1}\le \al< 1. $$
    Statement ({\rm{i}}) is proved.
    Therefore, the function $f$ is a contraction near the fixed point $y_\ast$, as was to be proved.
    The proof of  statements  ({\rm{ii}}) and ({\rm{iv}}) is straightforward. 
    \end{proof}
        
  \begin{proof}[Proof of Theorem~{\rm \ref{thm:conv_iter}}]  Fix $x> 0$ and choose $y_1(x) = 1 - e^{-x}$.  
    It is not difficult to see that the sequence $\{ y_n\}_{n\ge 1}$ satisfies the inequality $0\le y_n\le 1$ for $n\ge 2$ if
    $y_1\in (0, 1]$. Monotonicity and boundedness of the sequence $\{y_n\}_{n\ge 1}$ implies that there exists the limit
    $$ y_\ast(x)=\lim_{n\to\iy} y_n(x). $$ This limit
     does not depend on the choice of $y_1$  because convergence to $y_\ast(x)$ is monotone $\forall x\ge 0$, 
    and near the fixed point the  mapping is a contraction. 
  The limit, $y_\ast(x)$, satisfies  (\ref{eq:gL_cr}) and we obtain that
  $y_\ast(x) = y_\be(x)$ for all $x>0$. Theorem~\ref{thm:conv_iter} is thus proved.
  \end{proof}
  
   \begin{figure}[htp] 
    \vspace{-180pt}
   \centering 
   \includegraphics[scale=0.68,angle=0]{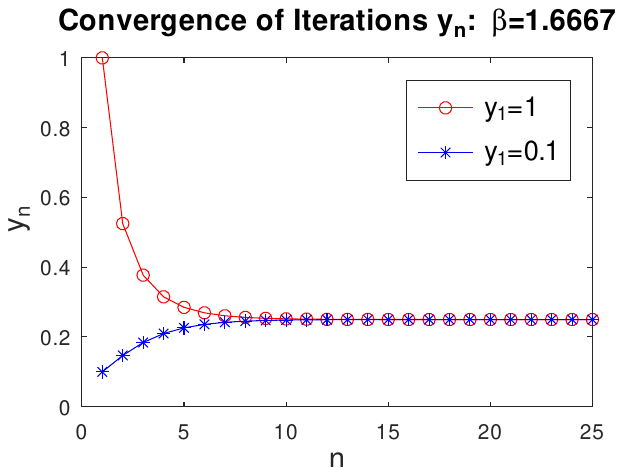}
    \vspace{-210pt}
  \caption{ Iterations of the mapping $\widehat\mcW$.}
  \label{fig:iterGL2}
\end{figure}
Convergence of iterations, $y_n(x)$, is illustrated in Figure~\ref{fig:iterGL2}. The integer parameter, $n$, is 
the iteration index, the argument $x=1.287$ in Figure~\ref{fig:iterGL2}.

The values of the function, $y_\be(x)$, also depend on the parameter, $\be>1$. If parameter $\be$ is not fixed, 
we have a function of two arguments, $y_\be(x)$.
\begin{proposition}\label{prop:monot_beta}
 The function $y_\be(x)$ is a monotone function of  $\be, (\be>1)$.
\end{proposition}
\begin{proof} Consider the inverse function $x_\be(y)$ corresponding to the  function  $y_\be(x)$. 
From the definition of the inverse function
(\ref{eq:inv_func}) it follows that $x_\be(y)$ is monotonically decreasing function of the parameter $\be$.
Then we immediately obtain that the inverse function to $x_\be(y)$  is monotonically increasing function of $\be$.
Proposition~\ref{prop:monot_beta} is thus proved.
\end{proof}
  \begin{figure}[htp] 
    \vspace{-200pt}
   \centering 
   \includegraphics[scale=0.68,angle=0]{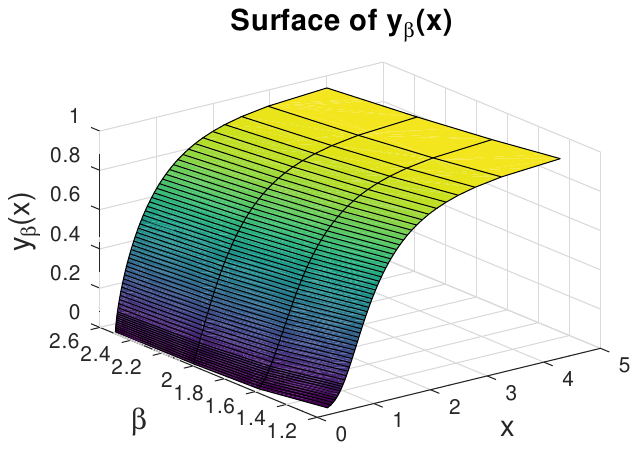}
    \vspace{-210pt}
  \caption{Surface of the function $y_\be(x)$.}
  \label{fig:surfY}
\end{figure}

Since the limit of iterations, $\{y_n\}_{n\ge 1}$, does not depend on the initial value, $y_1$, we obtain the following
\begin{corollary}
For any $\theta\in(0, 1]$ the function $y_\be(x)$ admits the following  continued exponential representation:
 \beq
     y_\be(x) =  \dots\(1 -\exp\( -x\cdot \( 1 -\exp\(-x\cdot (1- e^{-x\theta})^\al\) \)^\al \) \)^\al \dots  \label{eq:gL_cem}
 \eeq
 where $\al=\be^{-1}\in (0, 1)$.
\end{corollary}
  
  \section{Asymptotic behaviour of  function $y_\be(x)$.}\label{sec:asympt_gLf}
  In this section, we consider the asymptotic behaviour of $y_\be(x)$ for both small  and  large values of argument $x$.
  First,  we construct the following asymptotic expansion for  function $y_\be(x)$ in the right-hand 
  neighbourhood of $0$.
  \begin{proposition}
  The function, $y_\be(x)$, has the following asymptotic on the lower tail: 
  \beq 
         y_\be(x) = \frac{ x^{\frac{1}{\be-1}}}{  1 + \frac{ x^{\frac{\be+1}{\be -1}}}{2(\be-1)}  } + o\(x^{\frac{\be+2}{\be-1}}\)
         \quad\text{as $x\downarrow 0$}.  \label{eq:asymp_y}
   \eeq   
  \end{proposition}
  \begin{proof}
  We have 
  $$  1 - e^{-xy_\be} = xy_\be - \frac12 x^2 y_\be^2 +o(x^2). $$
  Let us rewrite Equation~(\ref{eq:ybe_def}) as
  $$
    y^\be = 1 - e^{-xy},    \quad x>0,   
 $$
 Then we obtain from the latter equation that
  \beq 
      \(y_\be(x)\)^{\be -1} = x\cdot\(1-\frac12 x y_\be(x) + o(x)\).   \label{eq:y_bet} 
 \eeq
 Let $\ga\bydef \be-1$. Since $\be>1$, the parameter $\ga$ is positive and hence, we get for 
 small $t$ that   $ \( 1 - t\)^{\ga}= 1- \ga t + o(t). $
  
  From (\ref{eq:y_bet}) and the latter asymptotic formula we obtain that $y_\be(x)$ satisfies the following relation:
  $$ y_\be(x) = x^{\frac{1}{\be-1}}\( 1 -\frac{1}{2(\be-1)} x y_\be(x) + o\(x^{\frac{1}{\be-1}}\)\).  $$
  Then we derive that
  $$ y_\be(x)\cdot \( 1 + \frac{ x^{\frac{\be}{\be -1}}}{2(\be-1)} \) =  x^{\frac{1}{\be-1}} +  o\(x^{\frac{2}{\be-1}}\), $$
  and, finally,
  $$ y_\be(x) = \frac{ x^{\frac{1}{\be-1}}}{  1 + \frac{ x^{\frac{\be}{\be -1}}}{2(\be-1)}  } +  o\(x^{\frac{2}{\be-1}}\). $$
  The proposition is thus proved.
  \end{proof}
  The asymptotics (\ref{eq:asymp_y}) of $y_\be(x)$, as $x\!\to\! 0$, is just a starting point. It can be improved as follows:
   \begin{proposition}\label{prop:asympt_y1}
  The function, $y_\be(x)$ admits  the following  asymptotic expansion as   $x\dar 0$:
  {\small
  \beq 
         y_\be(x) =  x^{\frac{1}{\be-1}} \cdot\(  1 - \frac{1}{2(\be-1)}  x^{\frac{\be}{\be-1}}  +  
         \frac{\be+8}{24(\be-1)^2}  x^{\frac{2\be}{\be-1}}  -  \frac{\be+3}{12(\be-1)^3}  x^{\frac{3\be}{\be-1}} +\dots   \)
          \label{eq:asymp_y1}
   \eeq   
   }
  \end{proposition}
  The proof of this proposition is deferred to the Appendix.
  \begin{corollary}
  The function $\hat z(x) \bydef  y_\be(x) \cdot x^{\frac{1}{1-\be}} $ satisfies the following limit relation
  $$ \lim_{x\dar 0} \hat z(x) =1.  $$
  \end{corollary}
  \begin{remark}
  Let us introduce a new variable, $t \bydef x^{\frac{\be}{\be-1}} $ and define 
  $$ z(t)\bydef \hat z(x). $$
  The function $z(t)$ satisfies the following transcendental equation:
  $$ z^{\be-1}(t) = \vf\(t\cdot z(t)\), $$
  where
  $$\vf(u)=\frac{1-e^{-u}}{u}. $$
  It is instructive to look at the graph of the function $z(t)$. 
 \begin{figure}[htp] 
    \vspace{-140pt}
   \includegraphics[width=0.9\textwidth]{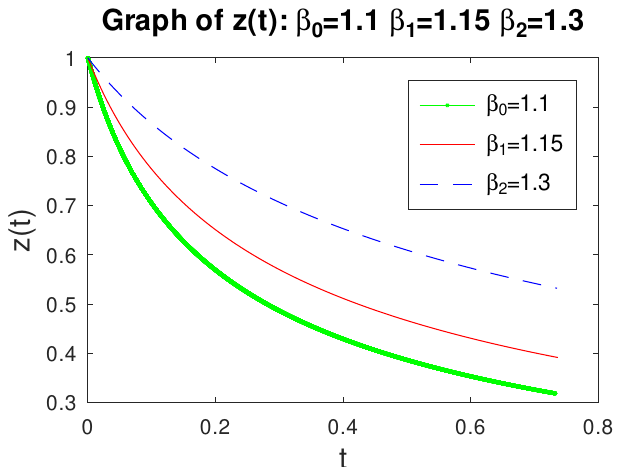}
    \vspace{-138pt}
  \caption{Graph of the function $z(t)$.}
  \label{fig:ztas}
\end{figure}
In Figure~\ref{fig:ztas} the graph of the function $z(t)$ is shown for the following three particular values of $\be$:
$\be_1=1.1$, $\be_2=1.15$, $\be_3=1.3$. 
All three graphs, obviously, start at the point $(0, 1)$. The function $z(t)$ is downward convex, positive and monotonically 
decreasing.

\subsection*{Conjecture}
  \noindent  {\it The function {\rm $z(t)$, $(t\ge 0)$}, is analytical.}
\end{remark}
 
 Our next step is the derivation of the two-sided  bounds for the function $y_\be(x)$. 
  \noindent Numerical comparison of the approximations (\ref{eq:asymp_y}) and (\ref{eq:asymp_y1}) and the lower bound, $y_L(x)$, introduced in Proposition~\ref{prop:bounds_y},  is discussed in Section~\ref{sec:numerics}.  
  
  \begin{proposition}\label{prop:bounds_y}
   For $x>0$ and $\be>1$,
  the function $y_\be(x)$ satisfies the following two-sided inequality:
  \beq
      \(1 - e^{-x}\)^{\frac{1}{\be-1}} < y_\be(x)<   \(1 - e^{-x}\)^{\frac{1}{\be}} \label{eq:ineq}
   \eeq
        \end{proposition}
        \begin{proof}
        It follows from the definition of the inverse function that
        $$ x = y_\be^{\be-1} + \frac12  y_\be^{2\be-1} + \frac13  y_\be^{3\be-1} + \dots $$
        Since $0<y_\ast<1$, we have
        $$ x < y_\be^{\be-1} + \frac12  y_\be^{2\be-2} + \frac13  y_\be^{3\be-3} + \dots $$
        The latter inequality can be written as
        $$ x < -\log\( 1 - y_\be^{\be-1}\). $$
        Therefore,
        $$ y_\be(x) >\( 1- e^{-x}\)^{\frac{1}{\be-1}}.  $$
        On the other hand, since
        $$ x = -\frac{\log(1-y_\be^\be)}{y_\ast} \quad\text{and} \quad 0<y_\be<1, $$
        we obtain that
        $$ x >  -\log(1-y_\be^\be) $$
      The latter inequality is equivalent to 
      $ y_\be(x) < \( 1- e^{-x}\)^{\frac{1}{\be}}.  $
      The upper bound, $y^U(x)= \( 1- e^{-x}\)^{\al}$, and the lower bound,
       $ y_L(x)= \( 1- e^{-x}\)^{\frac{\al}{1-\al}}$, hence established. 
        \end{proof}
        
   \begin{figure}[htp] 
    \vspace{-185pt}
   \centering 
   \includegraphics[scale=0.68,angle=0, left]{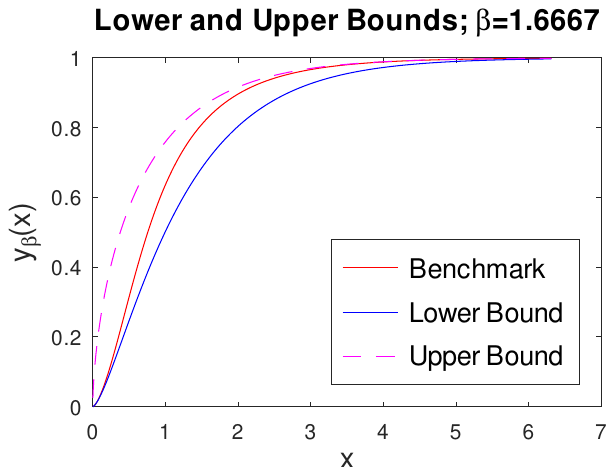}
    \vspace{-210pt}
  \caption{ The function $y_\be(x)$ and the bounds (\ref{eq:ineq}).}
  \label{fig:bounds}
\end{figure}

    In Figure~\ref{fig:bounds} the graph of the function, $y_\be(x)$, is displayed together with the upper
    and lower bounds established in    Proposition~\ref{prop:bounds_y}. 
    The lower bound captures  the asymptotic behaviour of the function $y_\be$ as $x\dar 0$
    while the upper   bound captures the asymptotic behaviour of $y_\be(x)$ as $x\to\iy$.
    
    It is also not difficult to see that   the ratio of the bounds approaches $1$ at infinity
    \beq
         \lim_{x\to\iy}  \frac{y^U(x)}{y_L(x)}=1. \label{eq:ratio_bounds}
      \eeq   
    Indeed, the ratio $\frac{y^U(x)}{y_L(x)}$, can be written as
    $$
           \frac{y^U(x)}{y_L(x)}=    \(1 + \frac{1}{e^x - 1}\)^\kappa, 
    $$
       where $ \kappa\bydef  \frac{\al^2}{1-\al}>0. $
       The latter relation implies  the validity of  (\ref{eq:ratio_bounds}).  
       
       \section{Numerical examples}\label{sec:numerics} 
       \subsection{Roots of Equation~(\ref{eq:e_pr_ds})}
       The problem of the computation of the extinction probability can be addressed by using several numerical methods.
       
       The first approach is related to solution of  Equation~(\ref {eq:e_pr_ds}), considered in Section~\ref{sec:ext_prob},
       by bisection method.  An important technical issue is due to the fact for the values of     parameter $\be$ close 
       to $1$ this method 
       often presents a difficulty in  finding the  interval containing the root of the equation. 
       This situation is illustrated in Figure~\ref{fig:bisect} where the graph of the function, $\mcF_\be(x)$, 
        satisfying Equation~(\ref{eq:mcF_be}), is shown for
       three different values of parameter $\be$: $\be_1=1.25 $,  $\be_2=1.6 $ and $\be_3=1.95$.
       
 The positive root of the equation, marked by the star in Figure~\ref{fig:bisect} is moving towards $0$ as 
 $\be\dar 1$.  
            
   \begin{figure}[htp] 
    \vspace{-150pt}
   \centering 
   \includegraphics[width=0.98\textwidth, left]{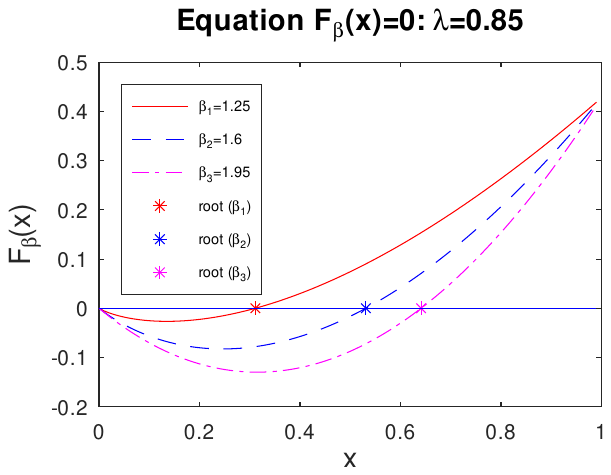}
    \vspace{-170pt}
  \caption{ Roots of Equation $\mcF_\be(x)=0$.}
  \label{fig:bisect}
\end{figure}
Moreover,  the minimal value of the function tends to $0$ monotonically, namely,
$ \inf_{x>0}\limits \mcF_\be(x)\downarrow 0  $ as $\be\downarrow 1 $. 
For this reason, the root of the equation, $x_0(\be)\dar 0$, as $\be\dar 1$.
Such a behaviour represents certain computational difficulties for the  bisection method. The numerical experiments
with MatLab
demonstrate  that  the interval, $[a, b]$, such that 
$\mcF_\be(a)\cdot \mcF_\be(b)<0$, remains undetermined for $\be<1.1$, due to computer finite precision.

\subsection{Comparison of approximations}\label{sec:comp_approx}
A much more efficient solution can be obtained by using the asymptotic approximations, derived 
in Section~\ref{sec:asympt_gLf}. Recall that the lower bound used as the asymptotic approximation of the function,
$y_\be(x)$, is $ y_L(x)= \( 1- e^{-x}\)^{\frac{\al}{1-\al}}$. 

%
The first and the second approximations, $y_\be^{(1)}(x)$ and $y_\be^{(2)}(x)$ of the function $y_\be(x)$, defined 
by Equations~(\ref{eq:asymp_y})
and (\ref{eq:asymp_y1}), respectively, and the lower bound, $y_L(x)$, are compared with the benchmark approximation of
$y_\be(x)$. The methodology of comparison is founded on a fixed set, $Y=\{y_1, y_2, \dots, y_n\}$, of values of
function, where $y_j$ are close to $0$. 
Given $\be>1$, we compute the benchmark approximation, $x_j=x_\be(y_j)$, $(j=1, 2, \dots, n)$,
and the distances $d_i(x) =\abs{ y_\be^{(i)}(x) - y_\be(x)}$, $(i=1, 2)$, and  $d_L(x) = \abs{ y_L(x) - y_\be(x)}$
for each $x=x_j$. 

\noindent In Figure~\ref{fig:asymptD3} we compare  logarithms of the distances, $\log\(d_i(x) \)$, $(i=1,2)$ and 
$\log\( d_L(x )\)$. Parameter $\be=1.25$. The set $Y$ is an arithmetic  progression,
$$ Y=\{ y_k \}_{k \ge 1}, \quad y_k = y_1 + \Delta \cdot (k-1), $$
where $y_1 = 0.0001$, $\max_{k\ge 1}\limits y_k=0.15$, $\Delta =0.0001$. The results shown in Figure~\ref{fig:asymptD3}
demonstrate advantage of the approximation $y_\be^{(2)}(x)$.

  \begin{figure}[htp] 
    \vspace{-150pt}
   \centering 
   \includegraphics[width=0.98\textwidth, left]{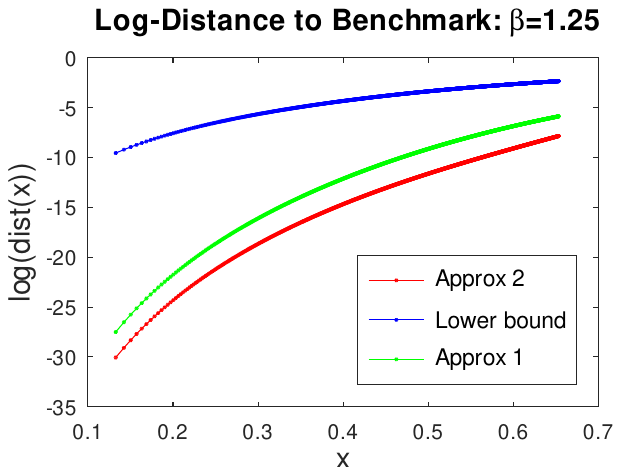}
    \vspace{-170pt}
  \caption{ Log-distance to benchmark.}
  \label{fig:asymptD3}
\end{figure}
Similar results are obtained for the other values of parameter $\be$, see Figure~\ref{fig:asymptD4}.
In that Figure, parameter $\be$ takes two values, $\be_1=1.25$ and $\be_2=1.35$.

 \begin{figure}[htp] 
    \vspace{-165pt}
   \centering 
   \includegraphics[width=0.98\textwidth, left]{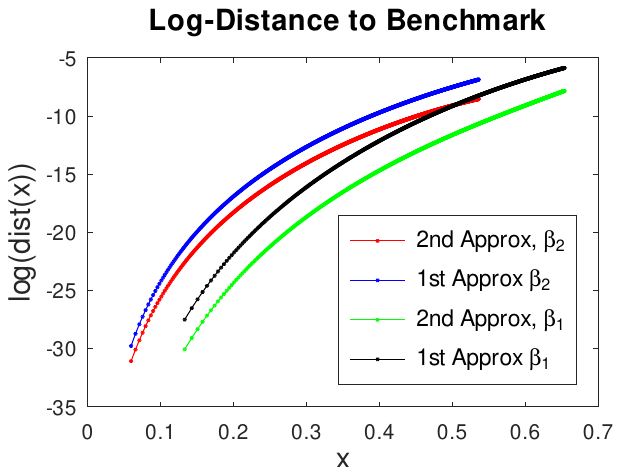}
    \vspace{-170pt}
  \caption{ Log-distance to benchmark: $\be_1=1.25, \be_2=1.35$.}
  \label{fig:asymptD4}
\end{figure}
In both cases the numerical results demonstrate advantage of  approximation (\ref{eq:asymp_y1}).
 
\subsection{The number of iterations required}
       The iterations described in Theorem~\ref{thm:conv_iter} converge independently of the choice of the
       initial approximation, $y_1$. However, the number of iterations, $\mcN(\vep)$,
        required for a specified accuracy $\vep>0$, does depend on $y_1$ as well as on $\be$ and $x$.
        Let us define $\mcN$ more precisely.
        \begin{definition}
        Given $x>0$, the required number of iterations, $\mcN(\vep, x)$, is defined as follows:
        $$ \mcN(\vep, x) \bydef \min_{n\ge 1}\abs{ y_n - y_\be(x) }\le\vep. $$
        \end{definition}
             
        Let us compare the following strategies for the choice of $y_1$:
        \bei
        \item[1. ] $y_1^{(1)}(x)= \(1 - e^{-x}\)^{\frac{1}{\be}} $ for all $x>0$.
        \item[2. ] $y_1^{(2)}(x) =  \(1 - e^{-x}\)^{\frac{1}{\be-1}} $.
        \item[3. ] $y_1^{(3)}(x) = \frac12\cdot\(   \(1 - e^{-x}\)^{\frac{1}{\be-1}}   +  \(1 - e^{-x}\)^{\frac{1}{\be}}  \)$.
        \eei
       The first strategy always takes the initial approximation to coincide with the upper bound. The second strategy
       takes the initial approximation to be equal to the lower bound for the solution.
       The third strategy takes the mid-point as the initial approximation.
          
         \begin{figure}[htp] 
    \vspace{-156pt}
 \includegraphics[width=0.98\textwidth, left]{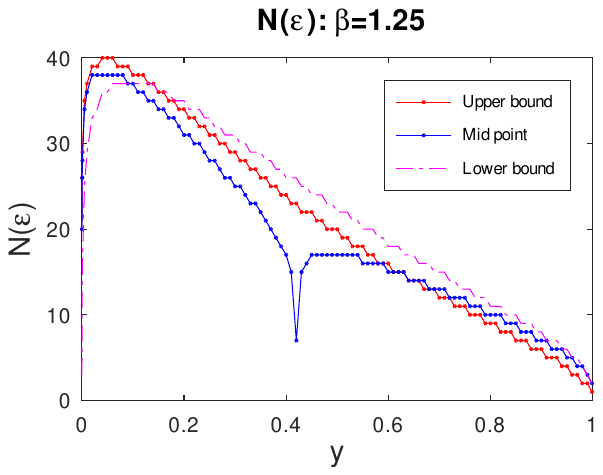}
    \vspace{-170pt}
  \caption{ Required number of iterations, $\mcN(\vep)$.}
  \label{fig:N_eps}
\end{figure}
In Figure~\ref{fig:N_eps}, the above three strategies are compared for $\vep=10^{-5}$. Parameter $\beta=1.25$
 in the numerical experiment related to Figure~\ref{fig:N_eps} is described in this section.

It is convenient to change the argument $x$ in the definition of the required number of iterations to $y_\be(x)$. In this 
case, the second argument of the function $\mcN(\vep, y)$ will belong to the interval $[0, 1]$.

 While the argument $y=y_\be(x)$ is small enough, $0 < y\le 0.15$, the initial approximation, $y_1^{(2)}(x)$, dominates
 the other strategies. In the mid-range case, $0.15 < y \le 0.6$, the mid-point strategy, $y_1^{(3)}(x)$, provides a
  faster convergence than the remaining two strategies. 
  
  For the large values of the argument $x$, such that $y\ge 0.6$, where $y=y_\be(x)$,  the first strategy 
  is the most efficient. This fact is consistent with the property of the upper bound, $y^U(x)$, to 
  provide the asymptotic  approximation of $y_\be(x)$ as $x\to\iy$.
     
       \section{Definition and properties of the random variable $\xi_\be$}\label{sec:rand_var}
       In the previous section we demonstrated that for each $\be>1$ the function $y_\be(x)$ is the cdf of an
       absolutely
       continuous random variable  denoted by $\xi_\be$.  Despite  the fact that the cdf of
       $\xi_\be$ is not known explicitly, many characteristics of $\xi_\be$ can still be found in closed form.
       In the next subsection, we compute the moments  of this random variable.
       \subsection{Computation of the moments, $\E[\xi_\be^n]$.}
       \begin{definition}
           We say that the random variable $\xi_\be$ has the generalized Lambert distribution if
       $$\Pp\(\xi_\be\le x\)=y_\be(x), \quad \forall x\ge 0. $$
       \end{definition}
    \begin{proposition}\label{prop:moments}   
    The $n$th moment of the r.v. $\xi_\be$ is
    \beq
     \E[ \xi_\be^n ] = (-1)^n \al\cdot \int_0^1 \frac{\log^n(1-t)}{t^{(n-1)\al +1}} \d t,
       \label{eq:nth_mom}
       \eeq
       where  $\al=1/\be$.
       \end{proposition}
    \begin{proof}
    Let $\xi_\be$ be defined on a certain probability space $\trp$ and let $\mcU=\mcU([0, 1])$ be a random variable on  
    the same probability space, $\trp$, having  uniform $[0, 1]$ distribution. The random 
  variable $\xi_\be$ can be written as follows:
\beq
       \xi_\be = -\frac{\log\(1-\mcU^\be\)}{\mcU}. \label{eq:xi_be}
   \eeq 
     Then we obtain from (\ref{eq:xi_be}) that the moments of the r.v. $\xi_\be$ satisfy the relation
     $$  \E[ \xi_\be^n ] = \int_0^1 (-1)^n s^{-n} \cdot \log^n\!\!\(1-s^\be\) \d s, \quad n=1, 2, \dots. $$
     Changing the  variables
     $ t \bydef s^\be, $ in the above integral ascertains the validity of
     (\ref{eq:nth_mom}).
    \end{proof}
    In Figure~\ref{fig:log_moments} the first 6 log-moments, $\log\( \E[\xi_\be^n]\)$ of the random variable 
    $\xi_\be$ are shown for $\be=1.2$. The theoretical results are obtained from (\ref{eq:nth_mom}). These results are
     \begin{figure}[htp] 
    \vspace{-140pt}
   \centering 
   \includegraphics[width=0.9\textwidth]{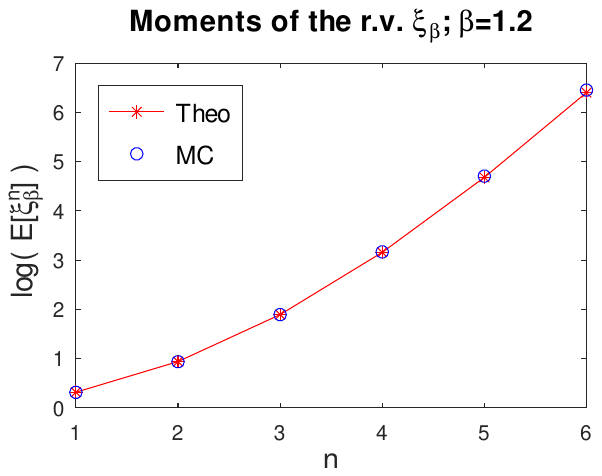}
    \vspace{-145pt}
  \caption{ The log-moments of $\! \xi_\be$:  $\!\log\!\( \E[\xi_\be^n]\), n=1, 2, \dots, 6$.}
  \label{fig:log_moments}
\end{figure}
compared with that obtained by Monte Carlo method.

    The following statement  stipulates integrability of the function $1-y_\be(x)$ and demonstrates unexpected
    connection to the Riemann  zeta function, $\zeta(x)$.
  \begin{proposition}\label{prop:int}
  The function $y_\be(x)$ satisfies the following relation:
  \beq
      \int_0^\iy \(1-y_\be(x)\) dx = \frac{\pi^2}{6}\cdot\frac{1}{\be}.  \label{eq:integr1}
   \eeq
\end{proposition}
  \begin{proof}
  Since the random variable $\xi_\be\ge 0$,
   its first moment can be obtained by integrating the corresponding survival function: 
   $$\E[\xi_\be] =  \int_0^\iy \(1-y_\be(x)\) dx. $$
      On the other hand, from (\ref{eq:nth_mom}) we have
  \beq
      \E[\xi_\be]=-\int_0^1 \frac{\log\(1-s^\be\)}{s} \d s. \label{eq:Exi}
   \eeq
    Denote $v\bydef s^\be$ and note that  for $0\le s<1$, 
    $$ \log\( 1 - v\) =-\sum_{n=1}^\iy \frac{v^n}{n}. $$
    Then from (\ref{eq:Exi}) we derive
    \begin{eqnarray*}
        \E[\xi_\be] &=& \int_0^1 \sum_{n=1}^\iy \frac{s^{\be n-1}}{n} \d s \\
        &=&  \frac1\be \cdot \sum_{n=1}^\iy \frac{1}{n^2}.
        \end{eqnarray*}
        Thus, formula (\ref{eq:integr1}) is derived.
  \end{proof}
 Obviously, we can write the first moment of $\xi_\be$ as $\E[\xi_\be]=\zeta(2) \cdot \al$.
   \begin{proposition}\label{prop:sec_mom}
  The second moment of the random variable, $\xi_\be$, satisfies the relation 
  \beq
     \E[\xi_\be^2] =\frac{1}{\be}\cdot  \int_0^1 t^{-1 -\al}\log^2(1-t)\d t.  \label{eq:m2}
   \eeq
   The standard deviation of $\xi_\be$ is
   \beq
        \si(\xi_\be)=\sqrt{ \al\cdot \int_0^1 t^{-1 -\al}\log^2(1-t)\d t - \frac{\pi^4}{36}\cdot\al^2 }. \label{eq:std_xi}
      \eeq  
\end{proposition}
  \begin{proof}
  The proof of (\ref{eq:m2}) follows from Equation~(\ref{eq:nth_mom}). The formula for the standard deviation, $\si(\xi_\be)$,
  follows from (\ref{eq:m2}) and  the relation $\E[\xi_\be]=\frac{\pi^2}{6} \cdot \al$. \end{proof}
  
  In Figure~\ref{fig:mom12}, the first moment, $E(\al) =\E[\xi_{1/\al}]$ and the standard deviation, 
  $\si(\al)=\si(\xi_{1/\al})$ are shown for $\al\in (0, 1)$. The analytical results are accompanied by the Monte Carlo
  valuation of the first moment of $\xi_\be$ and the standard deviation of this random variable. 
   
 \begin{figure}[htp] 
    \vspace{-140pt}
   \includegraphics[width=0.9\textwidth]{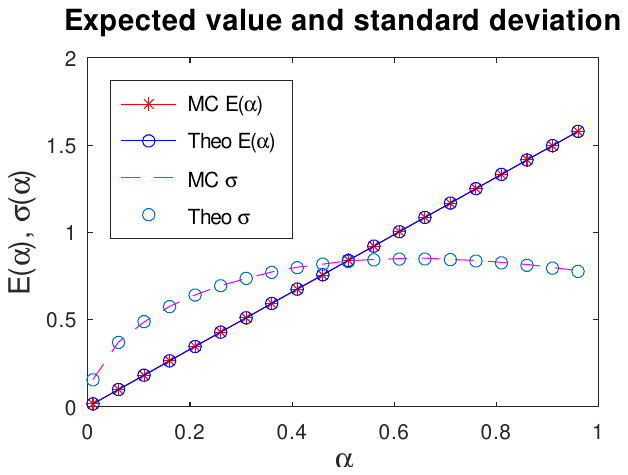}
    \vspace{-138pt}
  \caption{Expectation and standard deviation of $\! \xi_\be$.}
  \label{fig:mom12}
\end{figure}
The theoretical value of the expectation, $\E[\xi_\be]$, and its MC estimator form a straight line in 
Figure~\ref{fig:mom12}. The standard deviation, $\si(\al)$, and its MC estimator exhibit a non-linear, non-monotone
dependency on $\al$.
\subsection{The moment problem for  $\xi_\be$.}
The classical  moment problem \cite{Ahi},  \cite{Krein}, \cite{Schmu},  deals with existence and uniqueness of a finite, 
non-negative measure  $\mu$ on $\R^1$ such
that a given sequence of real numbers, $m_1$, $m_2$, $\dots$, $m_n$, $\dots$ satisfying the relation
$$ m_n = \int_{\R^1} x^n \d\mu(x), \quad n=0, 1, 2, \dots.  $$
In  this Section, we consider the case of the probability measure $\mu$ corresponding to the distribution of the random
variable $\xi_\be$:
$$  \mu(A) = \Pp\( \xi_\be\in A\), \quad A\,\,\text{is a measurable subset of $\R^1$.}  $$
In this case, the support of the measure $\mu$ is $\R^1_+$. 

Existence of such probability measure is obvious in our case. The uniqueness problem can be formulated as follows:
\bei
\item[$ $]
{\it Given a sequence of moments, $\{ M_n\}_{n\ge 1}$, is there another  measure, $\mu^\ast$,
different from the probability
measure $\mu$ defined by the cumulative distribution function $y_\be(x)=\Pp(\xi_\be\le x)$, such that the 
$n$th moment corresponding to measure
$\mu^\ast$ is $M_n$ for all integer $n\ge 1$?}
\eei
If the solution to the moment problem is unique, we say that the moment problem is  determinate. Otherwise
the moment problem is said to be indeterminate.
\begin{proposition}\label{prop:uniq_mom_probl}
Let the moments, $\{M_n\}_{n\ge 1}$, be defined by Equation~{\rm (\ref{eq:nth_mom})}:
$$
     M_n= (-1)^n \al\cdot \int_0^1 \frac{\log^n(1-s)}{s^{(n-1)\al +1}} \d s.
$$     
Then there exists a unique probability measure, $\mu$, on $\R^1_+$ such that
$$ \int_0^\iy x^n \d \mu(x) = M_n, \quad n=0, 1, \dots. $$
\end{proposition}
The proof of Proposition~\ref{prop:uniq_mom_probl} is based on the following technical statement on the
rate of growth of the moments, $M_n$.
\begin{lemma}\label{lem:mom_Mn}
The sequence $M_n$  exhibits  the following asymptotic relation:
\beq
    M_n =c_\be \cdot n! + {\mathcal{O}}\(\frac{n!}{2^n}\) \quad \text{as $n\to \iy$}, \label{eq:asympt_M_n}
  \eeq
  where the constant $c_\be$ satisfies the inequalities $\frac{1}{\be} < c_\be < \frac{1}{\be - 1}.$
  \end{lemma}
  \noindent Proof of Lemma~\ref{lem:mom_Mn} is deferred to the Appendix.
\begin{proof}[Proof of Proposition~{\rm \ref{prop:uniq_mom_probl}}]
Let us prove that the  measure $\mu$ is determinate.
    We would like to verify the following sufficient condition for the moments   of the
    probability measure concentrated on $\R^1_+$ (see \cite{Krein}, \cite{Schmu}, \cite{Shir} ):
    \beq
     \limsup_{n\to\iy} \frac{M_n^{1/n}}{n}< \iy \label{eq:mom_cond}
    \eeq 
    From (\ref{eq:asympt_M_n}) we derive that as $n\to\iy$
    $$ M_n = c_\be \sqrt{2\pi n} \cdot \( \frac{n}{e}\)^n + o(n!). $$  
    Then we find that
    $$ \limsup_{n\to\iy}  \frac{M_n^{1/n}}{n} = {\mathcal{O}}(1). $$
    The latter asymptotic equation implies (\ref{eq:mom_cond}), as was to be proved.\end{proof}
    \begin{remark}
    One can also use  Carleman's test (see \cite{Krein}, \cite{Schmu}, \cite{Shir} ) to verify that the sequence
    of moments, $\{M_n\}_{n\ge 1}$ is determinate.
    \end{remark}

\subsection{Moments, generating function and identities}
We conclude this section with a few remarks on the connection between the moments of the r.v. $\xi_\be$ and the 
generating functions of several remarkable sequences including the sequence $\{ \zeta(n)\}_{n\ge 2}$, where
$\zeta(\cdot)$ is the Riemann zeta function.

Let us start with the second moment, $\E[\xi_\be^2]$. From (\ref{eq:m2}) we have
\beq  \E[\xi^2_{1/\al}] = \al \cdot \mcI(\al),  \label{eq:E2_alf}  \eeq
where
$$  \mcI(\al)=  \int_0^1 t^{-1 -\al}\log^2(1-t)\d t.  $$
\begin{proposition}\label{prop:int_mom}
The integral, $\mcI(\al)$, satisfies the relation
\beq
    \mcI(\al) = \sum_{k=1}^\iy \sum_{m=1}^\iy \frac1k\cdot \frac1m \cdot \frac{1}{k+m-\al}, \qquad 0<\al <1.
       \label{eq:m3}
\eeq
\end{proposition}
Proposition~\ref{prop:int_mom} is proved in the Appendix.

Let us now represent the integral, $\mcI(\al)$, as the  generating function of a  harmonic                                                                                                                                                                                                                                                                                                                                                                                                              sequence. Denote by $H_n$,  the sum of the harmonic series
$$H_n=1 + \frac12 + \frac13 + \dots + \frac1n, \quad n=1, 2, \dots, $$
and for  $n=0$ we define $H_0=0$. Obviously, we have 
$$ H_n = H_{n-1} + \frac1n, \quad n=1, 2, \dots. $$
Let us now define the sequence
\beq
     A_\ell \bydef 2\cdot \sum_{n=1}^\iy \frac{H_{n-1}}{n^{\ell+2}}, \quad \ell=0, 1, 2, \dots.  \label{eq:A_ell}
  \eeq   
\begin{proposition}\label{prop:gen_fun_A}
The elements of the sequence $\{ A_\ell\}_{\ell\ge 1}$ satisfy the relation
 \beq
     A_\ell=  \sum_{k=1}^\iy \sum_{m=1}^\iy \frac1k\cdot \frac1m \cdot \frac{1}{(k+m)^{\ell + 1}},
  \quad \ell=0, 1, 2, \dots.  \label{eq:ident1}
  \eeq
The integral $\mcI(\al)$ is the generating function of the sequence $\{ A_\ell\}_{\ell\ge 1}$:
\beq
   \mcI(\al) =  \sum_{\ell=0}^\iy A_\ell \al^{\ell}. \label{eq:gf_Al}
  \eeq  
   \end{proposition}
The proof of Proposition~\ref{prop:gen_fun_A} is deferred to the Appendix.

The double sums similar to (\ref{eq:ident1}) are usually called {\it Tornheim series or Tornheim double sums}. 
They find applications in many
 areas of mathematics and physics, see  \cite{Kub},  \cite{Torn}.
 
\begin{remark}
The sequence $\{A_\ell\}_{\ell\ge 0}$ can also be expressed in terms of the Riemann zeta function as
follows 
$$ A_\ell = 2 \sum_{n=1}^\iy  \frac{H_n}{n^{\ell+2}} - 2\zeta(\ell+3). $$
\end{remark}
The coefficient, $A_0$, determines  the asymptotic behaviour of the second moment as $\al\dar 0$.
\begin{proposition}\label{prop:A0}
\beq
    A_0=2\zeta(3) =  2.404138\dots .  \label{eq:A0}
    \eeq
    \end{proposition}
    The proof of this proposition is deferred to the Appendix.
    \begin{corollary}
        The variance of the r.v. $\xi_\be$ admits  the following asymptotics as $\al\downarrow 0$: 
        $$ \si^2(\xi_\be) = A_0\cdot\al + o(\al). $$
    \end{corollary}
    \begin{proof}
    The statement of the corollary immediately follows from Equations (\ref{eq:E2_alf}) and (\ref{eq:gf_Al}).
    \end{proof}
    Computation of the coefficients, $A_1$, $A_2$, ..., $A_N$ requires computation of the harmonic sums 
    $$ \mcS_k = \sum_{n=1}^\iy \frac{H_n}{n^k},\quad  k=2, 3, \dots. $$
    \begin{corollary}[compare with \cite{AliD},  \cite{Flaj}]
    $$ $$
    \begin{itemize}
    \item[i.] 
    The coefficient $A_1$ satisfies the relation
    $$ A_1 = 2 \sum_{k=1}^\iy \frac{H_k}{k^3} - 2 \zeta(4)=\frac{\pi^4}{180}.  $$
    \item[ii.] The harmonic sum
    $$ \sum_{k=1}^\iy \frac{H_k}{k^3}=\frac{\pi^4}{72}. $$
    \end{itemize}
    \end{corollary}
    Numerous representations of such character  for  harmonic sums can also be found in \cite{Flaj}.
    
    Analysis of the higher moments of $\xi_\be$ leads to relations generalizing properties  of 
    harmonic double sums. This line of research will be considered in our subsequent work.  
  
{\bf Acknowledgment.} VVV appreciates hospitality of the Fields
Institute and York University.

\bibliographystyle{plain}
\bibliography{Lambert}

\appendix
\section{Proof of Technical Results}
\begin{proof}[Proof of Proposition~{\rm \ref{prop:asympt_y1}}]
  Let us represent the function, $y_\be$, as 
  \beq   y_\be(x) \bydef x^{\frac{1}{\be-1} } \cdot \hat z(x).   \label{eq:y_z} \eeq
  Then, obviously,
  $$ \hat z(x) =  y_\be(x) \cdot x^{\frac{1}{1-\be} } .  $$
  From Equation~(\ref{eq:asymp_y}) it follows that the limit
  $$ \lim_{x\dar 0} \hat z(x) =1.  $$
It is convenient to introduce the variable $t\bydef x^{\frac{\be}{\be -1} }$. 
Denote by $z(t)\bydef \hat z(x)$. Then we obtain that the function $z(t)$
satisfies the implicit equation
\beq
    z^{\be-1}(t) = 1 - \frac12 z(t)\cdot t + \frac16 (z(t)\cdot t)^2 -\frac1{24} (z(t)\cdot t)^3 + \dots \label{eq:z_t}
\eeq
Let us represent the function $z(t)$ in the following form:
\beq 
    z(t)\bydef 1 + \sum_{n=1}^\iy z_n t^n.  \label{eq:z_expan}
 \eeq   
Then substituting (\ref{eq:z_expan}) into (\ref{eq:z_t}), we recursively derive the first coefficients in  Expansion (\ref{eq:z_expan}):
$$
     z_1 = -\frac{1}{2(\be -1)}, \quad z_2 = \frac{\be+8}{24(\be-1)^2}, \quad 
     z_3 = -\frac{\be + 3}{12(\be -1)^3}.
$$
The proposition is thus proved.
 \end{proof}
\begin{proof}[Proof of Lemma~{\rm \ref{lem:mom_Mn}}]
We have to prove that the moments $M_n$ have the following asymptotics: 
$$
    M_n =c_\be\cdot n! + o\({n!}\) \quad \text{as $n\to \iy$},  
  $$
  where the constant $c_\be$ satisfies the inequality $\frac{1}{\be} < c_\be < \frac{1}{\be - 1}.$
  Since the support of the measure $\mu$ is $\R^1_+$, the $n^{th}$ moment, $M_n$,  can be evaluated as follows:
$$
    M_n = n \int_0^\iy x^{n-1} \bar F(x) \d x,   
$$
where the survival function $\bar F(x)=\mu([x, +\iy))$. Therefore,
\beq M_n =n \int_0^\iy x^{n-1} \bar y_\be(x) \d x,   \label{eq:mom_tail} \eeq
where $\bar y_\be(x) = 1 - y_\be(x)$. Combining (\ref{eq:ineq}) and (\ref{eq:mom_tail}) we derive the inequalities
\beq  \int_0^\iy \!x^{n-1}\! \(1 - (1 - e^{-x})^{1/\be}\) \!\d x < \!\frac{M_n}{n}\! < 
         \int_0^\iy \!x^{n-1}\! \(1 - (1 - e^{-x})^{\frac{1}{\be-1}}\)\! \d x. \label{eq:ineq_Mn}
     \eeq    
 Next, consider the integral
 $$ \mcJ(\ga)\bydef  n \int_0^\iy x^{n-1} \(1 - (1 - e^{-x})^{1/\ga}\) \d x, \quad \ga>0.  $$
 Taking into account that
 $$  (1 - e^{-x})^{1/\ga} = 1 -\frac1\ga e^{-x} + \frac{1-\ga}{2\ga^2} e^{-2x} + o\(e^{-2x}\) $$
 as $x\to\iy $, we obtain the following asymptotic relation as $n\to\iy$
 \beq 
     \mcJ(\ga) = \frac1\ga n \Gamma(n) - \frac{n\cdot \Gamma(n)}{2^n} \frac{1-\ga}{ \ga^2} + o\(\frac{n!}{2^n}\), 
 \label{eq:asympt_J}
 \eeq
 where $\Gamma(n) $ is the Gamma function. Then from (\ref{eq:ineq_Mn}) and (\ref{eq:asympt_J}) we obtain
 that
 $$ \limsup_{n\to\iy} \frac{M_n}{n!} =c_\be, \quad \frac1{\be}\le c_\be \le\frac{1}{\be-1}. $$
 The lemma is thus proved.
\end{proof}

\begin{proof}[Proof of Proposition~{\rm \ref{prop:int_mom}}]
Recall the representation of the function $\log(1-t)$ by the following convergent series:
$$-\log(1-t) = \sum_{k=1}^\iy \frac{t^k}{k}, \quad 0\le t < 1. $$
Hence, the integral $\mcI(\al)$ can be written as follows:
$$  \mcI(\al)=  \int_0^1 t^{-1 -\al} \sum_{k=1}^\iy \frac{t^k}{k} \cdot \sum_{m=1}^\iy  \frac{t^m}{m} \d t.  $$
The latter relation implies that
\begin{eqnarray*}
      \mcI(\al) &=&  \int_0^1 t^{-1 -\al} \sum_{k=1}^\iy \frac{t^k}{k} \cdot \sum_{m=1}^\iy  \frac{t^m}{m} \d t \\  
      &=& \sum_{k=1}^\iy \frac{1}{k} \cdot \sum_{m=1}^\iy \frac{1}{m} \int_{0}^1 t^{k+m-\al-1} \d t \\
      &=&  \sum_{k=1}^\iy \sum_{m=1}^\iy \frac{1}{k\cdot m\cdot (k+m-\al)}.
   \end{eqnarray*}   
The proposition is thus proved.
\end{proof}
\begin{proof}[Proof of Proposition~{\rm \ref{prop:gen_fun_A}}]
Consider the double sum in  (\ref{eq:m3}). We have that for  $0<\al < 1$, 
$$
    \frac{1}{k+m-\al} = \frac1{k+m} \cdot \sum_{\ell=0}^\iy \frac{\al^\ell}{(k+m)^\ell}.
$$
Then we derive that
\begin{eqnarray*}
    \mcI(\al) &=& \sum_{k=1}^\iy \sum_{m=1}^\iy \frac1k\cdot \frac1m \cdot \frac{1}{k+m-\al}  \\
    &=& \sum_{\ell=0}^\iy \al^\ell \sum_{k=1}^\iy \sum_{m=1}^\iy \frac1k\cdot \frac1m \cdot \frac{1}{(k+m)^{\ell + 1}}.
     \end{eqnarray*}   
    The latter relation implies (\ref{eq:ident1}). Now, consider  the inner double sum,
    $$ A_\ell= \sum_{k=1}^\iy \sum_{m=1}^\iy \frac1k\cdot \frac1m \cdot \frac{1}{(k+m)^{\ell + 1}}. $$
Changing the summation index, $n=k+m$, we obtain that
  \begin{eqnarray*}
     A_\ell &=&  \sum_{k=1}^\iy \sum_{m=1}^\iy \frac1k\cdot \frac1m \cdot \frac{1}{(k+m)^{\ell + 1}} \\
    &=&  \sum_{n=2}^\iy \frac{1}{n^{\ell+2}} \sum_{k=1}^{n-1} \(\frac1k + \frac{1}{n-k}\).
 \end{eqnarray*}   
 It is obvious that the latter sum is equal to $2 H_{n-1}$. Therefore,
 $$  \mcI(\al) = 2\sum_{\ell=0}^\iy \al^\ell \sum_{n=2}^\iy \frac{H_{n-1}}{n^{\ell+2}}.  $$
 From the definition of the sequence $A_\ell$ (see Equation~(\ref{eq:A_ell})) we obtain that
 $$ \mcI(\al) = \sum_{\ell=0}^\iy A_\ell \al^{\ell }, $$ as was to be proved.
 \end{proof}
 
 \begin{proof}[Proposition~{\rm \ref{prop:A0}}]
 Let us prove that 
$ A_0=2\zeta(3) $. From~(\ref{eq:ident1}) we have 
   $$ A_0=\sum_{k\ge 1}  \sum_{m\ge 1} \frac{1}{k\cdot m\cdot (k+m)}. $$
   Then we derive that
   \begin{eqnarray*}
       \sum_{k\ge 1}  \sum_{m\ge 1} \frac{1}{k\cdot m\cdot (k+m)} &=&
       \sum_{n=2}^\iy  \sum_{k=1}^{n-1}\frac{1}{n} \cdot \frac{1}{k\cdot (n-k)} \\
       &=&
       \sum_{n=2}^\iy  \sum_{k=1}^{n-1}\frac{1}{n^2} \(\frac{1}{k} +\frac1{ (n-k)}\) \\
       &=&
       2\sum_{n=1}^\iy  \frac{1}{n^2} \cdot \(H_{n} -\frac1n\).
   \end{eqnarray*}  
   On the other hand,
    \begin{eqnarray*}
       \sum_{k\ge 1}  \sum_{m\ge 1} \frac{1}{k\cdot m\cdot (k+m)} &=&
       \sum_{k\ge 1}  \sum_{m\ge 1} \frac{1}{k^2}\cdot\(\frac1m - \frac1{k+m}\) \\ &=&
       \sum_{k\ge 1}   \frac{1}{k^2}\cdot H_k.
       \end{eqnarray*}
       Thus, we obtain that
       $$  \sum_{k= 1}^\iy   \frac{1}{k^2}\cdot H_k =2\cdot \sum_{k= 1}^\iy   \frac{1}{k^3}
= 2 \zeta(3). $$
Formula (\ref{eq:A0}) is hence derived.
\end{proof}
 \end{document}